\documentclass[11pt, reqno]{amsart} 
\usepackage{amssymb,amscd,amsfonts,amsbsy}
\usepackage{latexsym}
\usepackage{exscale}
\usepackage{amsmath,amsthm,amsfonts}
\usepackage{mathrsfs}
\usepackage{xcolor} 
\usepackage[colorlinks=true,linkcolor=blue,citecolor=red,urlcolor=red, backref=page]{hyperref} 
\usepackage{esint} 
\usepackage{stmaryrd}
\usepackage{pifont}

\usepackage[utf8]{inputenc}

\calclayout

\allowdisplaybreaks

\newtheorem{theorem}[equation]{Theorem}
\newtheorem{proposition}[equation]{Proposition}
\newtheorem{lemma}[equation]{Lemma}

\theoremstyle{definition}

\numberwithin{equation}{section} 

\def\supp{\operatorname{supp}}

\def\BMO{\operatorname{BMO}}
\def\VMO{\operatorname{VMO}}
\def\XMO{\operatorname{XMO}}
\def\CMO{\operatorname{CMO}}
\def\Re{\operatorname{Re}}
\def\inf{\operatornamewithlimits{inf\vphantom{p}}}

\def\loc{\operatorname{loc}}

\def\C{\mathscr{C}}
\def\CC{\mathbb{C}}

\def\N{\mathbb{N}}

\def\R{\mathbb{R}}

\def\V{\mathscr{V}}

\def\Rn{\mathbb{R}^n}

\newcommand\restr[2]{\ensuremath{\left.#1\right|_{#2}}}
\DeclareMathOperator*{\esssup}{ess\,sup}

\begin{document}

\title{The $\mathrm{CMO}$-Dirichlet problem for elliptic systems in the upper half-space}

\author{Mingming Cao}
\address{Mingming Cao\\
Instituto de Ciencias Matem\'aticas CSIC-UAM-UC3M-UCM\\
Con\-se\-jo Superior de Investigaciones Cient{\'\i}ficas\\
C/ Nicol\'as Cabrera, 13-15\\
E-28049 Ma\-drid, Spain} \email{mingming.cao@icmat.es}

\thanks{The author is supported by Spanish Ministry of Science and Innovation through the Juan de la Cierva-Formaci\'{o}n 2018 (FJC2018-038526-I), through the ``Severo Ochoa Programme for Centres of Excellence in R\&D'' (CEX2019-000904-S), and through PID2019-107914GB-I00, and by the Spanish National Research Council through the ``Ayuda extraordinaria a Centros de Excelencia Severo Ochoa'' (20205CEX001)}

\date{June 1, 2022}




\begin{abstract}
We prove that for any second-order, homogeneous, $N \times N$ elliptic system $L$ with constant complex coefficients in $\mathbb{R}^n$, the Dirichlet problem in $\mathbb{R}^n_+$ with boundary data in $\mathrm{CMO}(\mathbb{R}^{n-1}, \mathbb{C}^N)$ is well-posed under the assumption that $d\mu(x', t) := |\nabla u(x)|^2\, t \, dx' dt$ is a strong vanishing Carleson measure in $\mathbb{R}^n_+$ in some sense.  This solves an open question posed by Martell et al. \cite{M+4}. The proof relies on a quantitative Fatou-type theorem, which not only guarantees the existence of the pointwise nontangential boundary trace for smooth null-solutions satisfying a strong vanishing Carleson measure condition, but also includes a Poisson integral representation formula of solutions along with a characterization of $\mathrm{CMO}(\mathbb{R}^{n-1}, \mathbb{C}^N)$ in terms of the traces of solutions of elliptic systems. Moreover, we are able to establish the well-posedness of the Dirichlet problem in $\mathbb{R}^n_+$ for a system $L$ as above in the case when the boundary data belongs to $\mathrm{XMO}(\mathbb{R}^{n-1}, \mathbb{C}^N)$, which lines in between $\mathrm{CMO}(\mathbb{R}^{n-1}, \mathbb{C}^N)$ and $\mathrm{VMO}(\mathbb{R}^{n-1}, \mathbb{C}^N)$. Analogously, we formulate a new brand of strong Carleson measure conditions and a characterization of $\mathrm{XMO}(\mathbb{R}^{n-1}, \mathbb{C}^N)$ in terms of the traces of solutions of elliptic systems. 
\end{abstract}

\maketitle

\section{Introduction}\label{sec:intro}
This paper is devoted to studying the Dirichlet problem for second-order elliptic systems with complex coefficients in the upper half-space with data in $\CMO$ and $\XMO$ spaces. To be more specific, we introduce some notation to elaborate on the actual setting.  

Fix $n, N \in \N$ with $n \geq 2$ and consider a second-order, homogeneous, constant complex coefficients, $N \times N$ system 
\begin{equation}\label{eq:Lu}
Lu := \big(a_{jk}^{\alpha\beta}\partial_j\partial_k u_\beta \big)_{1 \leq \alpha \leq N}, 
\end{equation} 
when acting on a $\mathscr{C}^2$ vector-valued function $u=(u_\beta)_{1\leq\beta\leq N}$ defined in an open subset of $\Rn$, where $a_{jk}^{\alpha\beta}\in\mathbb{C}$ for every $j,k \in \{1, \dots, n\}$ and $\alpha,\beta\in\{1, \dots, N\}$. Here and elsewhere, we use the convention of summation over repeated indices. We also assume that $L$ is elliptic, in the sense that there exists a constant $\kappa_0>0$ such that the following Legendre-Hadamard condition holds: 
\begin{equation}\label{eq:elliptic} 
\begin{array}{c}
\Re \bigl[a_{jk}^{\alpha\beta}\xi_j\xi_k\overline{\zeta_\alpha}\zeta_\beta \bigr]
\geq\kappa_0 |\xi|^2 |\zeta|^2\,\,\text{ for every}
\\[10pt]
\xi=(\xi_j)_{1\leq j\leq n} \in \Rn
\,\,\text{ and }\,\,\zeta=(\zeta_\alpha)_{1 \leq \alpha \leq N} \in \mathbb{C}^N.
\end{array}
\end{equation} 
In the scalar case (i.e. $N=1$), elliptic operators include the Laplacian $\Delta=\sum_{j=1}^n \partial_j^2$ or, more generally, operators of the form ${\rm div}(A\nabla)$, where $A=(a_{jk})_{1 \leq j, k \leq n}\in \mathbb{C}^{n \times n}$ satisfies the scalar version of \eqref{eq:elliptic}, that is,
\[
\inf_{\xi\in \mathbb{S}^{n-1}}{\rm Re}\,\big[a_{rs}\xi_r\xi_s\bigr]>0,
\]  
where $\mathbb{S}^{n-1}$ stands for the unit sphere in $\Rn$. Regarding the case $N>1$, an example of an elliptic system is the complex version of the Lam\'{e} system of elasticity in $\Rn$, given by 
\begin{equation*}
L:=\mu \Delta + (\lambda + \mu) \nabla{\rm div},
\end{equation*} 
where the constants $\lambda,\mu\in{\mathbb{C}}$ (called Lamé moduli in the literature) satisfy  
\[
\Re\mu>0 \quad\text{ and }\quad \Re(2\mu+\lambda)>0,
\]
which are indeed equivalent to \eqref{eq:elliptic}. While the Lam\'{e} system is symmetric, we stress that the results in this paper require no symmetry for the systems involved. 

We are interested in showing well-posedness for Dirichlet boundary value problems for $L$ in \eqref{eq:Lu}--\eqref{eq:elliptic} in the upper half-space. With this purpose in mind, given $n \geq 2$, we denote the upper half-space in $\Rn$ as 
\begin{equation*}
\R^n_{+} := \{(x', t) \in \Rn: \, x' \in \R^{n-1}, t>0\}.
\end{equation*}
We also identify the boundary $\partial\R^n_{+}$ with $\R^{n-1}$ via $\partial \R^n_{+} \ni (x', 0) \equiv x' \in \R^{n-1}$. The cone with vertex at $x'\in\R^{n-1}$ and aperture $\kappa>0$ is given by
\begin{equation*}
\Gamma_{\kappa}(x'):=\{(y',t) \in \R^n_{+}: |x'-y'|<\kappa t\}.
\end{equation*}
Given a vector-valued function $u: \R^n_{+} \to \mathbb{C}^N$, we define its nontangential boundary trace (whenever it is meaningful) as  
\begin{equation*}
\big(u\big|^{{}^{\kappa-{\rm n.t.}}}_{\partial\R^n_{+}}\big)(x'):=
\lim_{\Gamma_{\kappa}(x') \ni y \to (x',0)} u(y), \quad x' \in \R^{n-1},
\end{equation*}
and the nontangential maximal function of $u$ as
\begin{equation*}
\mathcal{N}_{\kappa}u(x') := \esssup\{|u(y)|:y\in\Gamma_{\kappa}(x')\},
\quad x' \in \R^{n-1}.
\end{equation*}

Let us introduce $\BMO(\R^{n-1}, \mathbb{C}^N)$, the John-Nirenberg space of vector-valued functions of bounded mean oscillations in $\R^{n-1}$, as the collection of $\mathbb{C}^N$-valued functions $f=(f_{\alpha})_{1 \le \alpha \le N}$ with components in $L^1_{\loc}(\R^{n-1})$ satisfying 
\[
\|f\|_{\BMO(\R^{n-1}, \mathbb{C}^N)} 
:= \sup_{Q \subseteq \R^{n-1}} \fint_{Q} |f(x') - f_Q| dx' < \infty,
\]
where $f_Q$ denotes the average value of $f$ on the cube $Q \subset \R^{n-1}$. In order to introduce significant spaces of functions of vanishing mean oscillations, we let $\C_c^{\infty}(\R^{n-1}, \mathbb{C}^N)$ denote the space of all smooth $\mathbb{C}^N$-valued functions in $\R^{n-1}$ with compact support. Define $\CMO(\R^{n-1}, \mathbb{C}^N)$ as the closure of $\C_c^{\infty}(\R^{n-1}, \mathbb{C}^N)$ in $\BMO(\R^{n-1}, \mathbb{C}^N)$. Additionally, the space $\CMO(\R^{n-1}, \mathbb{C}^N)$ is endowed with the norm of $\BMO(\R^{n-1}, \mathbb{C}^N)$. Setting
\begin{align*}
\mathscr{B}^1(\R^{n-1}, \mathbb{C}^N)  
&:= \Big\{f \in \C^1(\R^{n-1}, \mathbb{C}^N) \cap \BMO(\R^{n-1}, \mathbb{C}^N): 
\lim_{|x| \to \infty} |\nabla f(x)|=0 \Big\},
\\ 
\mathscr{B}^{\infty}(\R^{n-1}, \mathbb{C}^N) 
&:=\Big\{f \in \C^{\infty}(\R^{n-1}, \mathbb{C}^N) \cap \BMO(\R^{n-1}, \mathbb{C}^N): 
\lim_{|x| \to \infty} |\partial^{\alpha} f(x)|=0, \, \forall \alpha \in \N^n\Big\}, 
\end{align*}
we define $\mathrm{XMO}(\R^{n-1}, \mathbb{C}^N)$ as the closure of $\mathscr{B}^{\infty}(\R^{n-1}, \mathbb{C}^N)$ in $\BMO(\R^{n-1}, \mathbb{C}^N)$, with the norm of $\BMO(\R^{n-1}, \mathbb{C}^N)$. We mention that the $\XMO$ space in the scalar-valued case was introduced in \cite{TX} to study the compactness of commutators. Moreover, following the proof in \cite{TXYY, TX}, one can prove that 
\begin{align}\label{eq:CVBMO}
\CMO(\R^{n-1}, \mathbb{C}^N) 
\subsetneq \mathrm{XMO}(\R^{n-1}, \mathbb{C}^N) 
\subsetneq \VMO(\R^{n-1}, \mathbb{C}^N) 
\subsetneq \BMO(\R^{n-1}, \mathbb{C}^N), 
\end{align}
and a characterization of $\mathrm{XMO}(\R^{n-1}, \mathbb{C}^N)$: 
\begin{align}
\text{$\mathrm{XMO}(\R^{n-1}, \mathbb{C}^N)$ is the closure of $\mathscr{B}^1(\R^{n-1}, \mathbb{C}^N)$ in $\BMO(\R^{n-1}, \mathbb{C}^N)$}. 
\end{align}

Next, we turn to the Carleson measure conditions. Given a continuously differentiable function $u$ in $\R^n_+$, we set 
\begin{align*}
\|u\|_{\mathcal{C}(\R^n_+)} := \sup_{Q \subset \R^{n-1}} 
\bigg(\frac{1}{|Q|} \iint_{T_Q} |\nabla u(x)|^2\, t \, dx' dt \bigg)^{\frac12}, 
\end{align*}
where the supremum runs over all cubes $Q$ in $\R^{n-1}$ and $T_Q := Q \times (0, \ell(Q))$. Consider also the following quantities: 
\begin{align*}
\beta_1(u) &:= \lim_{r \to 0^+} \sup_{Q \subset \R^{n-1}: \ell(Q) \le r} 
\bigg(\frac{1}{|Q|} \iint_{T_Q} |\nabla u(x)|^2\, t \, dx' dt \bigg)^{\frac12}, 
\\
\beta_2(u) &:= \lim_{r \to \infty} \sup_{Q \subset \R^{n-1}: \ell(Q) \ge r} 
\bigg(\frac{1}{|Q|} \iint_{T_Q} |\nabla u(x)|^2\, t \, dx' dt \bigg)^{\frac12}, 
\\
\beta_3(u) &:= \lim_{r \to \infty} \sup_{Q \subset \R^{n-1} \setminus Q(0,r)} 
\bigg(\frac{1}{|Q|} \iint_{T_Q} |\nabla u(x)|^2\, t \, dx' dt \bigg)^{\frac12},  
\\
\beta'_3(u; Q) &:= \lim_{|x_Q| \to \infty} 
\bigg(\frac{1}{|Q|} \iint_{T_Q} |\nabla u(x)|^2\, t \, dx' dt \bigg)^{\frac12}, \quad\text{ cube } Q \subset \R^{n-1}. 
\end{align*}
We then define three kinds of spaces of functions related to vanishing Carleson measures: 
\begin{align*}
\V_1(\R^n_+) &:= \{u: \|u\|_{\mathcal{C}(\R^n_+)}<\infty, \beta_1(u)=0\}, 
\\
\V_2(\R^n_+) &:= \{u: \|u\|_{\mathcal{C}(\R^n_+)}<\infty, \beta_1(u)=\beta'_3(u; Q)=0 \text{ for each cube } Q \subset \R^{n-1}\}, 
\\ 
\V_3(\R^n_+) &:= \{u: \|u\|_{\mathcal{C}(\R^n_+)}<\infty, \beta_1(u)=\beta_2(u)=\beta_3(u)=0\}. 
\end{align*}

The $L^p$-Dirichlet boundary value problem for $L$ as in \eqref{eq:Lu}--\eqref{eq:elliptic} in the upper half-space was first studied by Martell et al. \cite{MMMM16}, in which the Poisson kernel, an $N \times N$-valued function described in detail in Theorem \ref{thm:Poisson}, plays a pivotal role. Additionally, they also proved the well-posedness of Dirichlet problem with boundary data in Banach function spaces, or Hardy spaces, or Morrey spaces. After that, the same authors \cite{MMMM19} established the well-posedness of the $\BMO$-Dirichlet boundary value problem whenever $d\mu(x', t) := |\nabla u(x)|^2\, t \, dx' dt$ is a Carleson measure in $\mathbb{R}^n_+$, that is, $\|u\|_{\mathcal{C}(\R^n_+)}<\infty$. In a similar way, for the $\VMO$-Dirichlet problem, it requires that $d\mu(x', t) := |\nabla u(x)|^2\, t \, dx' dt$ is a vanishing Carleson measure in $\mathbb{R}^n_+$, which means $u \in \V_1(\R^n_+)$. Beyond that, Martell et al. \cite{MMM} proved well-posedness results for the Dirichlet problem in $\R^n_+$ with boundary data in generalized \"{o}lder and Morrey-Campanato. Recently, by means of Rubio de Francia extrapolation on weighted Banach function and modular spaces, Cao, Mar\'{i}n, and Martell \cite{CMM} showed the well-posedness of the Dirichlet problem for $L$ with the boundary data belonging to a general weighted Banach function spaces or a weighted modular space. As a consequence, one can obtain that the Dirichlet problem for such systems is well-posed for boundary data in Lebesgue spaces, variable exponent Lebesgue spaces, Lorentz spaces, Orlicz spaces, as well as their weighted versions. 

We are now ready to state our first main theorem. It concerns the well-posedness of the $\CMO$-Dirichlet problem in the upper half-space for systems $L$ as in \eqref{eq:Lu}. 
	
\begin{theorem}\label{thm:CMO}
Let $L$ be an $N \times N$ elliptic constant complex coefficient system as in \eqref{eq:Lu}--\eqref{eq:elliptic}. Then the $\CMO$-Dirichlet boundary value problem for $L$ in $\R^n_+$, namely
\begin{equation}\label{eq:CMO}
\begin{cases}
u \in \C^{\infty}(\R^n_+, \mathbb{C}^N), \\
Lu = 0 \text{ in } \R^n_+, \\
u \in \V_3(\R^n_+), \\
\restr{u}{\partial \R^n_+}^{\rm n.t.} = f \in \CMO(\R^{n-1}, \mathbb{C}^N) \text{ a.e. on } \R^{n-1},
\end{cases}
\end{equation} 
is well-posed. Moreover, the unique solution is given by 
\begin{align}\label{eq:CMO-solution}
u(x', t) := P_t^L * f(x'), \quad (x', t) \in \R^n_+, 
\end{align}
where $P^L$ denotes the Poisson kernel for $L$ in $\R^n_+$ from Theorem {\rm \ref{thm:Poisson}}, and satisfies the following  properties: 
\begin{align}
& \sup_{\varepsilon>0} \|u(\cdot, \varepsilon)\|_{\BMO(\R^{n-1}, \mathbb{C}^N)} 
\le C \|u\|_{\mathcal{C}(\R^n_+)},
\\
& f \in \CMO(\Rn, \mathbb{C}^N) \text{ if and only if } u \in \V_3(\R^n_+), 
\\
& C^{-1} \|f\|_{\BMO(\R^{n-1}, \mathbb{C}^N)} \le \|u\|_{\mathcal{C}(\R^n_+)} 
\le C \|f\|_{\BMO(\R^{n-1}, \mathbb{C}^N)}, 
\end{align} 
where the constant $C=C(n, L)$ depends only on the dimension $n$ and the ellipticity constant. 
\end{theorem}

Our second main theorem considers the well-posedness of the $\XMO$-Dirichlet problem in the upper half-space for systems $L$ as in \eqref{eq:Lu}. 

\begin{theorem}\label{thm:XMO}
Let $L$ be an $N \times N$ elliptic constant complex coefficient system as in \eqref{eq:Lu}--\eqref{eq:elliptic}. Then the $\mathrm{XMO}$-Dirichlet boundary value problem for $L$ in $\R^n_+$, namely
\begin{equation}\label{eq:XMO}
\begin{cases}
u \in \C^{\infty}(\R^n_+, \mathbb{C}^N), \\
Lu = 0 \text{ in } \R^n_+, \\
u \in \V_2(\R^n_+), \\
\restr{u}{\partial \R^n_+}^{\rm n.t.} = f \in \mathrm{XMO}(\R^{n-1}, \mathbb{C}^N) \text{ a.e. on } \R^{n-1},
\end{cases}
\end{equation} 
is well-posed. Moreover, the unique solution is given by 
\begin{align}\label{eq:XMO-solution}
u(x', t) := P_t^L * f(x'), \quad (x', t) \in \R^n_+, 
\end{align}
where $P^L$ denotes the Poisson kernel for $L$ in $\R^n_+$ from Theorem {\rm \ref{thm:Poisson}}, and satisfies the following  properties: 
\begin{align}
& \sup_{\varepsilon>0} \|u(\cdot, \varepsilon)\|_{\BMO(\R^{n-1}, \mathbb{C}^N)} 
\le C \|u\|_{\mathcal{C}(\R^n_+)},
\\
& f \in \XMO(\Rn, \mathbb{C}^N) \text{ if and only if } u \in \V_2(\R^n_+), 
\\
& C^{-1} \|f\|_{\BMO(\R^{n-1}, \mathbb{C}^N)} \le \|u\|_{\mathcal{C}(\R^n_+)} 
\le C \|f\|_{\BMO(\R^{n-1}, \mathbb{C}^N)}, 
\end{align} 
where the constant $C=C(n, L)$ depends only on the dimension $n$ and the ellipticity constant. 
\end{theorem}

The proof of Theorem \ref{thm:CMO} relies on a quantitative Fatou-type theorem below, which states the existence of the pointwise nontangential boundary trace for any smooth null-solution $u$ satisfying a Carleson measure condition, that is, $u \in \V_3(\R^n_+)$. It also gives a Poisson integral representation formula of the solution $u$. 

\begin{theorem}\label{thm:Fatou-CMO} 
Let $L$ be an $N \times N$ elliptic system with constant complex coefficients as in \eqref{eq:Lu}--\eqref{eq:elliptic} and let $P^L$ be the Poisson kernel for $L$ in $\R^n_+$ from Theorem {\rm \ref{thm:Poisson}}.  Then there exists a constant $C=C(n, L) \in (1, \infty)$ such that 
\begin{equation}\label{eq:u-CMO}
\left.
\begin{aligned}
& u \in \C^{\infty}(\R^n_+, \mathbb{C}^N), \\
& Lu=0 \text{ in } \R^n_+, \\ 
& u \in \V_3(\R^n_+),
\end{aligned}
\right\}
\, \, \Longrightarrow \, \, 
\left\{
\begin{aligned}
& \restr{u}{\partial \R^n_+}^{\rm n.t.} \text{ exists a.e. in } \R^{n-1}, \text{ lies in } \CMO(\R^{n-1}, \mathbb{C}^N), \\
& u(x', t) = (P_t^L*(\restr{u}{\partial \R^n_+}^{\rm n.t.}))(x') \quad\text{for all } (x', t) \in \R^n_+, \\
& C^{-1} \|u\|_{\mathcal{C}(\R^n_+)} 
\le \|\restr{u}{\partial \R^n_+}^{\rm n.t.}\|_{\BMO(\R^{n-1}, \mathbb{C}^N)} 
\le C \|u\|_{\mathcal{C}(\R^n_+)}. 
\end{aligned}
\right.
\end{equation}
Furthermore, the following characterization of $\CMO(\R^{n-1}, \mathbb{C}^N)$ holds:  
\begin{align}\label{eq:CMO-V3}
\CMO(\R^{n-1}, \mathbb{C}^N) 
= \big\{\restr{u}{\partial \R^n_+}^{\rm n.t.}: u \in {\rm LMO}(\R^n_+) \cap \V_3(\R^n_+) \big\}, 
\end{align}
where 
\begin{align}\label{eq:LMO}
{\rm LMO}(\R^n_+) :=\{u \in \C^{\infty}(\R^n_+, \mathbb{C}^N): Lu=0 \text{ in } \R^n_+, \, 
\|u\|_{\mathcal{C}(\R^n_+)} < \infty\}. 
\end{align}
\end{theorem}

Analogously, to show the uniqueness of Theorem \ref{thm:XMO}, we present a Fatou-type theorem as well. 

\begin{theorem}\label{thm:Fatou-XMO} 
Let $L$ be an $N \times N$ elliptic system with constant complex coefficients as in \eqref{eq:Lu}--\eqref{eq:elliptic} and let $P^L$ be the Poisson kernel for $L$ in $\R^n_+$ from Theorem {\rm \ref{thm:Poisson}}.  Then there exists a constant $C=C(n, L) \in (1, \infty)$ such that 
\begin{equation}\label{eq:u-XMO}
\left.
\begin{aligned}
& u \in \C^{\infty}(\R^n_+, \mathbb{C}^N), \\
& Lu=0 \text{ in } \R^n_+, \\ 
& u \in \V_2(\R^n_+),
\end{aligned}
\right\}
\, \, \Longrightarrow \, \, 
\left\{
\begin{aligned}
& \restr{u}{\partial \R^n_+}^{\rm n.t.} \text{ exists a.e. in } \R^{n-1}, \text{ lies in } \XMO(\R^{n-1}, \mathbb{C}^N), \\
& u(x', t) = (P_t^L*(\restr{u}{\partial \R^n_+}^{\rm n.t.}))(x') \quad\text{for all } (x', t) \in \R^n_+, \\
& C^{-1} \|u\|_{\mathcal{C}(\R^n_+)} 
\le \|\restr{u}{\partial \R^n_+}^{\rm n.t.}\|_{\BMO(\R^{n-1}, \mathbb{C}^N)} 
\le C \|u\|_{\mathcal{C}(\R^n_+)}. 
\end{aligned}
\right.
\end{equation}
Furthermore, the following characterization of $\XMO(\R^{n-1}, \mathbb{C}^N)$ holds:  
\begin{align}\label{eq:XMO-V2}
\XMO(\R^{n-1}, \mathbb{C}^N) 
= \big\{\restr{u}{\partial \R^n_+}^{\rm n.t.}: u \in {\rm LMO}(\R^n_+) \cap \V_2(\R^n_+) \big\}, 
\end{align}
where ${\rm LMO}(\R^n_+)$ is given in \eqref{eq:LMO}. 
\end{theorem}

This paper is organized as follows. Section \ref{sec:pre} collects useful background material and auxiliary results including the existence of the Poisson kernel and characterizations of $\CMO$ and $\XMO$ spaces. The proofs of the existence in Theorems \ref{thm:CMO} and \ref{thm:XMO}, for the $\CMO$-Dirichlet problem and the $\XMO$-Dirichlet problem, are carried out in Section \ref{sec:exist}. Finally, Section \ref{sec:unique} is devoted to showing Theorems \ref{thm:Fatou-CMO} and \ref{thm:Fatou-XMO}, which imply the uniqueness in the $\CMO$-Dirichlet problem and the $\XMO$-Dirichlet problem, respectively. 

\section{Preliminaries}\label{sec:pre}

For every elliptic system as in \eqref{eq:Lu}--\eqref{eq:elliptic}, there exists an associated Agmon-Douglis-Nirenberg Poisson kernel in $\R^n_+$  \cite{ADN1, ADN2}, see also \cite{KMR2, Sol}. This, \cite[Theorems 2.4 and 3.1]{MMMM16}, and \cite[Theorem 1.1]{M+5} (see also \cite[Theorem 3.2]{MMMM16}) allow us to formulate the following result:

\begin{theorem}\label{thm:Poisson}
Let $L$ be a second-order, homogeneous, constant complex coefficient, $N \times N$ elliptic system in $\R^n$ as in \eqref{eq:Lu}--\eqref{eq:elliptic}. Then there exists a matrix-valued function $P^L=(P^L_{\alpha\beta})_{1 \leq \alpha, \beta \leq N}: \R^{n-1} \to \mathbb{C}^{N \times N}$, called the Poisson kernel for $L$ in $\R^n_{+}$, satisfying the following properties: 

\begin{list}{\textup{(\theenumi)}}{\usecounter{enumi}\leftmargin=1cm \labelwidth=1cm \itemsep=0.2cm 
		\topsep=.2cm \renewcommand{\theenumi}{\alph{enumi}}}

\item\label{eq:Poissson-1} There exists $C \in (0, \infty)$ such that 
\begin{equation}\label{eq:Poisson-decay}
|P^L(x')|\leq\frac{C}{(1+|x'|^2)^{n/2}}, \qquad x' \in \R^{n-1}. 
\end{equation} 

\item\label{eq:Poissson-2} The function $P^L$ is Lebesgue measurable and 
\begin{equation}\label{eq:Identity-1}
\int_{\R^{n-1}} P^L(x') \, dx' = I_{N \times N},
\end{equation} 
where $I_{N \times N}$ denotes the $N \times N$ identity matrix. In particular, for every constant vector $C=(C_{\alpha})_{1 \le \alpha \le N} \in \mathbb{C}^N$ one has 
\begin{equation}\label{eq:Identity-2}
\int_{\R^{n-1}} \sum_{1 \le \beta \le N} (P^L_{\alpha \beta})_t (x'-y') C_{\beta}\, dy' 
= C_{\alpha}, \quad\forall (x', t) \in \R^n_+. 
\end{equation}

\item\label{eq:Poisson-3} If one sets 
\begin{align}\label{def:KL}
K^L(x',t) := P_t^L(x')= t^{1-n}P^L(x'/t), \quad \forall (x', t) \in \R^n_+,
\end{align}
then the function $K^L:=\big(K^L_{\alpha \beta}\big)_{1\leq\alpha,\beta\leq N}$ satisfies (in the sense of distributions) 
\begin{equation}\label{eq:Poisson-LK0}
LK^L_{\cdot\beta}=0\,\,\text{ in }\, \, \R^n_{+} \quad \text{ for every }\quad \beta \in \{1,\dots, N\},
\end{equation}
where $K^L_{\cdot\beta}:=\big(K^L_{\alpha\beta}\big)_{1\leq\alpha\leq N}$ is the $\beta$-th column in $K^L$. 

Moreover, $P^L$ is unique in the class of $\CC^{N \times N}$-valued functions defined in $\R^{n-1}$ and satisfying \eqref{eq:Poissson-1}--\eqref{eq:Poisson-3} above, and has the following additional properties: 

\item One has $P^L \in \C^\infty(\R^{n-1})$, and $K^L \in \C^\infty(\overline{\R^n_+} \setminus \{0\}\big)$. Consequently, \eqref{eq:Poisson-LK0} holds in a pointwise sense. 

\item There holds $K^L(\lambda x)=\lambda^{1-n} K^L(x)$ for all $x \in \R^n_+$ and $\lambda>0$. In particular, for each multi-index $\alpha \in \N_0^n$ there exists $C_{\alpha} \in (0, \infty)$ with the property that 
\begin{align}\label{eq:Poisson-smooth}
|\partial^{\alpha} K^L(x)| \le C_{\alpha} |x|^{1-n-|\alpha|}, \quad\forall x \in \overline{\R^n_+} \setminus\{0\}. 
\end{align}

\item Assume that $f=(f_\beta)_{1\le\beta\le N}: \R^{n-1} \to \mathbb{C}^N$ is a Lebesgue measurable function such that 
\begin{equation}\label{zdfgwseg}
\int_{\R^{n-1}} \frac{|f(x')|}{1+|x'|^n} dx' < \infty. 
\end{equation}
Then the function $u(x',t):=P^L_t\ast f(x')$ for every $(x', t) \in \R^n_+$ is meaningfully defined via an absolutely convergent integral, and for every aperture $\kappa>0$, satisfies 
\begin{equation}\label{eq:Poisson-u-prop} 
u \in \C^\infty(\R^n_{+}, \mathbb{C}^N),\quad\, 
Lu=0\,\,\text{ in }\,\,\R^n_{+},\quad\,\,
u\big|^{{}^{\kappa-{\rm n.t.}}}_{\partial\R^n_{+}} = f \text{ a.e.~on }\,\R^{n-1}, 
\end{equation}
and there is a constant $C \in (0, \infty)$ such that 
\begin{equation}\label{eq:Poisson-u-prop2} 
|f(x')|\leq\mathcal{N}_{\kappa}u(x')\leq C_{\kappa}\,M f(x'), \qquad \text{for a.e.~} x' \in \R^{n-1}.
\end{equation}

\end{list}
\end{theorem}

Let us present a characterization of $\XMO(\R^{n-1}, \mathbb{C}^N)$ and $\CMO(\R^{n-1}, \mathbb{C}^N)$. In the scalar-valued case, the characterization of $\XMO(\R^{n-1})$ was given in \cite[Theorem 1.2]{TXYY} with $p=1$. Following the proof of \cite[Theorem 1.2]{TXYY}, one can show a more general result with exponent $p \in [1, \infty)$ below. 

\begin{proposition}\label{pro:XMO}
Let $1 \le p<\infty$. Then $f \in \mathrm{XMO}(\R^{n-1}, \mathbb{C}^N)$ if and only if $f \in \BMO(\R^{n-1}, \mathbb{C}^N)$ satisfies the following three conditions:
\begin{align*}
\gamma_1(f) &:= \lim_{r \to 0} \sup_{Q \subset \R^{n-1}:\ell(Q) \leq r}
\bigg(\fint_Q |f(x')-f_Q|^p dx' \bigg)^{\frac1p} =0,
\\
\gamma'_3(f; Q) &:= \lim_{|x_Q| \to \infty} \bigg(\fint_{Q} |f(y')-f_Q|^p dy' \bigg)^{\frac1p} =0, 
\quad\text{for each cube } Q \subset \R^{n-1}. 
\end{align*}
\end{proposition}

It should be pointed out that the next result was first announced by Neri \cite{Neri} without proof, and explicitly shown by Uchiyama \cite{Uch} for $p=1$, by Deng et al. \cite{DDSTY} for $p=2$, and by Ding and Mei \cite{DM} for $1 \le p<\infty$. Moreover, the three limiting conditions below are mutually independent, see \cite{Daf}.

\begin{proposition}\label{pro:CMO}
Let $1 \le p<\infty$. Then $f \in \CMO(\R^{n-1}, \mathbb{C}^N)$ if and only if $f \in \BMO(\R^{n-1}, \mathbb{C}^N)$ satisfies the following three conditions:
\begin{align*}
\gamma_1(f) &:= \lim_{r \to 0} \sup_{Q \subset \R^{n-1}:\ell(Q) \leq r}
\bigg(\fint_Q |f(x')-f_Q|^p dx' \bigg)^{\frac1p} =0,
\\
\gamma_2(f) &:= \lim_{r \to \infty} \sup_{Q \subset \R^{n-1}:\ell(Q) \geq r}
\bigg(\fint_Q |f(x')-f_Q|^p dx' \bigg)^{\frac1p} =0,
\\
\gamma_3(f) &:= \lim_{r \to \infty} \sup_{Q \subset \R^{n-1} \setminus Q(0,r)}
\bigg(\fint_Q |f(x')-f_Q|^p dx' \bigg)^{\frac1p} =0. 
\end{align*}
\end{proposition}

\section{The existence of the solution}\label{sec:exist}
The goal of this section is to show the existence of the solution. To this end, we present some fundamental estimates. 

\begin{lemma}
Fix $\varepsilon>0$ arbitrary. Then there exists a constant $C_{n, \varepsilon} \in (0, \infty)$ such that for each function $f \in L^1_{\loc}(\R^{n-1}, \mathbb{C}^N)$ and each cube $Q \subset \R^{n-1}$ with center $x'_Q \in \R^{n-1}$, there holds 
\begin{align}\label{eq:fff}
\int_{\R^{n-1}} \frac{|f(y')-f_Q|}{(\ell(Q)+|y' - x'_Q|)^{n-1+\varepsilon}} dy' 
\le \frac{C_{n, \varepsilon}}{\ell(Q)^{\varepsilon}} \sum_{k=0}^{\infty} 2^{-k \varepsilon} \fint_{2^k Q} |f(y')-f_{2^k Q}| dy'. 
\end{align}
\end{lemma}

\begin{proof}
Fix $f \in L^1_{\loc}(\R^{n-1}, \mathbb{C}^N)$ and a cube $Q \subset \R^{n-1}$ with center $x'_Q \in \R^{n-1}$. We compute 
\begin{align}\label{eq:fff-1}
&\int_{\R^{n-1}} \frac{|f(y')-f_Q|}{(\ell(Q)+|y' - x'_Q|)^{n-1+\varepsilon}} dy' 
\nonumber\\ 
&\quad \le \int_{Q} \frac{|f(y')-f_Q|}{\ell(Q)^{n-1+\varepsilon}} dy' 
+ \sum_{k=0}^{\infty} \int_{2^{k+1}Q \setminus 2^k Q} \frac{|f(y')-f_Q|}{|y' - x'_Q|^{n-1+\varepsilon}} dy' 
\nonumber\\ 
&\quad \lesssim \ell(Q)^{-\varepsilon}  \fint_{Q} |f(y')-f_Q| dy' 
+ \ell(Q)^{-\varepsilon} \sum_{k=0}^{\infty} 2^{-k \varepsilon} \fint_{2^{k+1}Q} |f(y')-f_Q| dy'. 
\end{align}
Let us consider the second term. For any $k \in \N_0$ and $p \in [1, \infty)$, we have 
\begin{align}\label{eq:tele}
\bigg(\fint_{2^{k+1} Q} |f-f_Q|^p dy' \bigg)^{\frac1p}
&\le \bigg(\fint_{2^{k+1} Q} |f-f_{2^{k+1} Q}|^p dy' \bigg)^{\frac1p}
+ \sum_{j=0}^k |f_{2^j Q} - f_{2^{j+1} Q}|
\nonumber\\ 
&\le \bigg(\fint_{2^{k+1} Q} |f-f_{2^{k+1}Q}|^p dy' \bigg)^{\frac1p}
+ \sum_{j=0}^k \fint_{2^j Q} |f- f_{2^{j+1} Q}| \, dy' 
\nonumber\\ 
&\lesssim \sum_{j=0}^k \bigg(\fint_{2^{j+1} Q} |f- f_{2^{j+1} Q}|^p \, dy' \bigg)^{\frac1p}, 
\end{align}
which gives that 
\begin{align}\label{eq:fff-2}
\sum_{k=0} 2^{-k \varepsilon} \fint_{2^{k+1}Q} |f(y')-f_Q| dy'
&\lesssim \sum_{k=0}^{\infty} 2^{-k \varepsilon} \sum_{j=1}^{k+1} \fint_{2^j Q} |f(y')- f_{2^j Q}| \, dy' 
\nonumber\\ 
&\lesssim \sum_{j=1}^{\infty} \bigg(\sum_{k=j-1}^{\infty} 2^{-k \varepsilon} \bigg) \fint_{2^j Q} |f(y')- f_{2^j Q}| \, dy' 
\nonumber\\ 
&\lesssim \sum_{j=0}^{\infty} 2^{-j \varepsilon} \fint_{2^j Q} |f(y')- f_{2^j Q}| \, dy'. 
\end{align}
Therefore, \eqref{eq:fff} follows at once from \eqref{eq:fff-1} and \eqref{eq:fff-2}. 
\end{proof}

\begin{lemma}\label{lem:SIO}
Let $\theta : \R^n_+ \times \R^n_+ \to \CC^{N \times N}$ be a matrix-valued Lebesgue measurable function, with the property that there exist $\varepsilon \in (0, \infty)$ and $C \in (0, \infty)$ such that 
\begin{align}\label{eq:size}
|\theta(x', t; y')| \le \frac{C \, t^{\varepsilon}}{|(x'-y', t)|^{n-1+\varepsilon}}, \quad\forall (x', t) \in \R^n_+,\, y' \in \R^{n-1}, 
\end{align}
and the cancellation condition holds: 
\begin{align}\label{eq:cancel}
\int_{\R^{n-1}} \theta(x', t; y') dy' = 0 \in \CC^{N \times N}, \quad\forall (x', t) \in \R^n_+. 
\end{align}
In relation to the kernel $\theta$, one may then consider the integral operator $\Theta$ acting on arbitrary function $f \in L^1(\R^{n-1}, dx'/(1+|x'|^{n-1+\varepsilon}))^N$ according to (the absolutely convergent integral) 
\begin{align}\label{eq:Theta}
\Theta f(x', t) := \int_{\R^{n-1}} \theta(x', t; y') f(y') \, dy', \quad (x', t) \in \R^n_+. 
\end{align}
Then the following statements hold:  
\begin{list}{\textup{(\theenumi)}}{\usecounter{enumi}\leftmargin=1cm \labelwidth=1cm \itemsep=0.2cm 
		\topsep=.2cm \renewcommand{\theenumi}{\alph{enumi}}}

\item\label{list-1} If in addition assume that $\theta$ satisfies 
\begin{align}\label{eq:smooth}
|\nabla_{y'} \theta(x', t; y')| \le \frac{C \, t^{\varepsilon}}{|(x'-y', t)|^{n+\varepsilon}}, \quad\forall (x', t) \in \R^n_+,\, y' \in \R^{n-1}, 
\end{align}
then 
\begin{align}\label{eq:L2-bound}
\Theta : L^2(\R^{n-1}, \mathbb{C}^N) \to L^2(\R^n_+, dx' dt/t)^N \text{ is bounded}. 
\end{align}

\item\label{list-2} Under the assumption \eqref{eq:L2-bound}, there exists a constant $C \in (0, \infty)$ such that for every $f \in L^1(\R^{n-1}, dx'/(1+|x'|^{n-1+\varepsilon}))^N$ and every cube $Q \subset \R^{n-1}$, 
\begin{align}\label{eq:TQ}
\bigg(\frac{1}{|Q|} \iint_{T_Q} |\Theta f(x', t)|^2 \frac{dx' dt}{t}\bigg)^{\frac12} 
\le C \sum_{k=0}^{\infty} 2^{-k \varepsilon} \bigg(\fint_{2^k Q} |f(x')-f_{2^k Q}|^2 \, dx' \bigg)^{\frac12}. 
\end{align}

\end{list}
\end{lemma}

\begin{proof}
It suffices to show part \eqref{list-2} since part \eqref{list-1} is proved in \cite[Proposition 3.2]{MMMM19}. Let $f \in L^1(\R^{n-1}, dx'/(1+|x'|^{n-1+\varepsilon}))^N$ and $Q \subset \R^{n-1}$ be a cube with the center $x'_Q$. The cancellation condition \eqref{eq:cancel} enables us to write 
\begin{align}\label{eq:TQ-12}
\bigg(\frac{1}{|Q|} \iint_{T_Q} |\Theta f(x', t)|^2 \frac{dx' dt}{t}\bigg)^{\frac12} 
= \bigg(\frac{1}{|Q|} \iint_{T_Q} |\Theta (f-f_Q)(x', t)|^2 \frac{dx' dt}{t}\bigg)^{\frac12} 
\le {\rm I} + {\rm II}, 
\end{align}
where 
\begin{align*}
{\rm I} &:= \bigg(\frac{1}{|Q|} \iint_{T_Q} |\Theta ((f-f_Q) \mathbf{1}_{4Q})(x', t)|^2 \frac{dx' dt}{t}\bigg)^{\frac12}, 
\\
{\rm II} &:= \bigg(\frac{1}{|Q|} \iint_{T_Q} |\Theta ((f-f_Q) \mathbf{1}_{\R^{n-1} \setminus 4Q})(x', t)|^2 \frac{dx' dt}{t}\bigg)^{\frac12}.
\end{align*}
To control ${\rm I}$, we use the assumption \eqref{eq:L2-bound} and \eqref{eq:tele} for $k=1$ and $p=2$ to get 
\begin{align}\label{eq:TQ-1}
{\rm I} 
\lesssim \bigg(\fint_{4Q} |f-f_Q|^2 dy' \bigg)^{\frac12}
\lesssim \bigg(\fint_{2Q} |f-f_{2Q}|^2 dy' \bigg)^{\frac12}
+ \bigg(\fint_{4Q} |f-f_{4Q}|^2 dy' \bigg)^{\frac12}. 
\end{align}
In order to estimate ${\rm II}$, we note that 
\[
|y'-x'| \simeq |y'-x'_Q| \gtrsim \ell(Q), \quad \text{ for all } x' \in Q, \, y' \in \R^{n-1} \setminus 4Q, 
\]
which together with \eqref{eq:size} and \eqref{eq:fff} yields 
\begin{align*}
|\Theta ((f-f_Q) \mathbf{1}_{\R^{n-1} \setminus 4Q})(x', t)|
&\lesssim t^{\varepsilon} \int_{\R^{n-1} \setminus 4Q} \frac{|f(y')-f_Q|}{(t+|x'-y'|)^{n-1+\varepsilon}} dy' 
\\
&\lesssim t^{\varepsilon} \int_{\R^{n-1}} \frac{|f(y')-f_Q|}{(\ell(Q)+|y' - x'_Q|)^{n-1+\varepsilon}} dy' 
\\
& \lesssim \frac{t^{\varepsilon}}{\ell(Q)^{\varepsilon}} \sum_{k=0}^{\infty} 2^{-k \varepsilon} \fint_{2^k Q} |f(y')-f_{2^k Q}| dy'. 
\end{align*}
Then, this implies 
\begin{align}\label{eq:TQ-2}
{\rm II} &\lesssim \bigg(\int_0^{\ell(Q)} \frac{t^{2\varepsilon}}{\ell(Q)^{2\varepsilon}} \frac{dt}{t}\bigg)^{\frac12} 
\sum_{k=0}^{\infty} 2^{-k \varepsilon} \fint_{2^k Q} |f(y')-f_{2^k Q}| dy' 
\nonumber \\ 
&\lesssim \sum_{k=0}^{\infty} 2^{-k \varepsilon} \fint_{2^k Q} |f(y')-f_{2^k Q}| dy' 
\nonumber \\ 
&\le \sum_{k=0}^{\infty} 2^{-k \varepsilon} \bigg(\fint_{2^k Q} |f(y')-f_{2^k Q}|^2 dy' \bigg)^{\frac12}. 
\end{align}
At this stage, we conclude \eqref{eq:TQ} from \eqref{eq:TQ-12}--\eqref{eq:TQ-2}. 
\end{proof}

The following result states the existence of the solution of the $\CMO$-Dirichlet problem \eqref{eq:CMO} and the $\XMO$-Dirichlet problem \eqref{eq:XMO}, respectively. 

\begin{proposition}\label{pro:existence} 
Let $L$ be an $N \times N$ elliptic system with constant complex coefficients as in \eqref{eq:Lu}--\eqref{eq:elliptic} and let $P^L$ be the Poisson kernel for $L$ in $\R^n_+$ from Theorem {\rm \ref{thm:Poisson}}. Select $f \in L^1(\R^{n-1}, dx'/(1+|x'|^n))^N$ and set 
\[
u(x', t) := P_t^L*f(x'), \quad (x', t) \in \R^n_+. 
\]
Then $u$ is meaningfully defined via an absolutely convergent integral and satisfies 
\begin{align}\label{eq:uuff-1}
u \in \C^{\infty}(\R^n_+, \mathbb{C}^N), \quad Lu=0 \text{ in } \R^n_+, \, \, \text{ and } \, \, 
\restr{u}{\partial \R^n_+}^{\rm n.t.} =f \text{ a.e. in } \R^{n-1}. 
\end{align}
In addition, the function $u$ enjoyes the following properties: 
\begin{align}
\label{eq:uuff-2} & f \in \XMO(\R^{n-1}, \mathbb{C}^N) \quad \Longrightarrow \quad u \in \V_2(\R^n_+), 
\\
\label{eq:uuff-3} & f \in \CMO(\R^{n-1}, \mathbb{C}^N) \quad \Longrightarrow \quad u \in \V_3(\R^n_+).
\end{align}
\end{proposition}

\begin{proof}
The estimate \eqref{eq:uuff-1} is shown in \cite[Proposition 3.1]{MMMM19}. To prove \eqref{eq:uuff-2}, fix $f \in \BMO(\R^{n-1}, \mathbb{C}^N)$, which along with \cite[Proposition 3.1]{MMMM19} gives that 
\begin{align}\label{eq:uuff-Car}
\text{$|\nabla u(x', t)|^2 \, d \, dx' dt$ is a Carleson measure in $\R^n_+$},  
\end{align}
that is, $\|u\|_{\mathcal{C}(\R^n_+)}<\infty$.

To proceed, write $K^L(x', t)=P_t^L(x')$ for every $(x', t) \in \R^n_+$. From \eqref{eq:Identity-2}, we see that 
\begin{align}\label{eq:grad-PL}
\int_{\R^{n-1}} \nabla^k [P_t^L(x'-y')] dy' =0, \quad\forall t>0, \, k \in \N. 
\end{align}
Moving on, for each $j \in \{1, \ldots, n\}$ and $\alpha, \beta \in \{1, \ldots, N\}$, set 
\[
\theta^j_{\alpha \beta}(x', t; y') := t \, \partial_j K^L_{\alpha \beta}(x'-y', t), \quad x', y' \in \R^{n-1}, \, t>0. 
\]
Observe that by \eqref{eq:Poisson-smooth} and \eqref{eq:grad-PL}, 
\begin{align}
\label{eq:theta-1} |\theta^j_{\alpha \beta}(x', t; y')| 
& = t |\partial_j K^L_{\alpha \beta}(x'-y', t)| \lesssim \frac{t}{|(x'-y', t)|^n}, 
\\
\label{eq:theta-2} |\nabla_{y'}\theta^j_{\alpha \beta}(x', t; y')| 
&\le t |\nabla^2 K^L_{\alpha \beta}(x'-y', t)| \lesssim \frac{t}{|(x'-y', t)|^{n+1}}. 
\end{align}
and 
\begin{align}\label{eq:theta-3}
\int_{\R^{n-1}} \theta^j_{\alpha \beta}(x', t; y') dy' 
= \int_{\R^{n-1}} t \, \partial_j K^L_{\alpha \beta}(x'-y', t) dy'
= t \, \partial_j \int_{\R^{n-1}}  K^L_{\alpha \beta}(y', t) dy'
=0. 
\end{align}
Writing $\Theta^j_{\alpha \beta}$ for the operator associated with the kernel $\theta^j_{\alpha \beta}$ as in \eqref{eq:Theta}. Granted this, from \eqref{eq:theta-1}--\eqref{eq:theta-3}, all hypotheses in Lemma \ref{lem:SIO} are verified for the matrix integral operators $\Theta^j:=(\Theta^j_{\alpha \beta})_{1\le \alpha, \beta \le N}$. In addition, 
\begin{align*}
\bigg(\frac{1}{|Q|} \iint_{T_Q} |\nabla u(x', t)|^2 t\, dx' dt\bigg)^{\frac12} 
&= \bigg(\frac{1}{|Q|} \iint_{T_Q} |t \nabla (P_t^L*f)(x')|^2 t\, dx' dt\bigg)^{\frac12} 
\\ \nonumber
&\le \sum_{j=1}^n \bigg(\frac{1}{|Q|} \iint_{T_Q} |t (\partial_j K^L(\cdot, t)*f)(x')|^2 t\, dx' dt\bigg)^{\frac12} 
\\ \nonumber
&= \sum_{j=1}^n \bigg(\frac{1}{|Q|} \iint_{T_Q} |\Theta^j f(x', t)|^2 \frac{dx' dt}{t}\bigg)^{\frac12}. 
\end{align*}
We then deduce from this and \eqref{eq:TQ} that 
\begin{align}\label{eq:ufQQ}
\bigg(\frac{1}{|Q|} \iint_{T_Q} |\nabla u(x', t)|^2 t\, dx' dt\bigg)^{\frac12} 
\lesssim \sum_{k=0}^{\infty} 2^{-k} \bigg(\fint_{2^k Q} |f-f_{2^k Q}|^2 \, dx' \bigg)^{\frac12}
=: \sum_{k=0}^{\infty} 2^{-k} a_k(Q). 
\end{align}

Now let us verify \eqref{eq:uuff-2} and \eqref{eq:uuff-3}.  Observe that 
\begin{align}\label{eq:supakQ}
\sup_{k \in \N_0} \sup_{Q \subset \R^{n-1}} a_k(Q) \lesssim \|f\|_{\BMO(\R^{n-1})}. 
\end{align}
Recall the definition of $\V_2(\R^n_+)$ and $\V_3(\R^n_+)$. To get \eqref{eq:uuff-2} and \eqref{eq:uuff-3}, by \eqref{eq:ufQQ}--\eqref{eq:supakQ} and Propositions \ref{pro:XMO} and \ref{pro:CMO},  it suffices to show that for any fix $k \in \N_0$,  
\begin{align}
\label{eq:rr-1} \gamma_1(f) =0 
&\quad\Longrightarrow\quad \lim_{r \to 0} \sup_{Q \subset \R^{n-1}:\ell(Q) \leq r} a_k(Q) = 0, 
\\
\label{eq:rr-2} \gamma_2(f) =0 
&\quad\Longrightarrow\quad \lim_{r \to \infty} \sup_{Q \subset \R^{n-1}:\ell(Q) \ge r} a_k(Q) =0, 
\\
\label{eq:rr-3} \gamma'_3(f; Q) =0 
&\quad\Longrightarrow\quad \lim_{|x_Q| \to \infty} a_k(Q) =0, \quad\text{for each cube } Q \subset \R^{n-1},  
\\
\label{eq:rr-23} \gamma_2(f) = \gamma_3(f) =0 &\quad\Longrightarrow\quad \lim_{r \to \infty} \sup_{Q \subset \R^{n-1} \setminus Q(0,r)} a_k(Q) =0. 
\end{align}
Indeed, \eqref{eq:rr-1}--\eqref{eq:rr-3} hold trivially. To show \eqref{eq:rr-23}, we fix $r>0$ and $Q \subset \R^{n-1} \setminus Q(0, r)$. The fact $\gamma_2(f)=\gamma_3(f)=0$ implies that for any given $\eta>0$, there exists $r_0=r_0(\eta)>0$ such that  
\begin{align}
\label{eq:gaaf-2} & \sup_{Q' \subset \R^{n-1}:\ell(Q') \geq r_0} 
\bigg(\fint_{Q'} |f(x')-f_{Q'}|^2 dx' \bigg)^{\frac12} < \eta,
\\
\label{eq:gaaf-3} & \sup_{Q' \subset \R^{n-1} \setminus Q(0, r_0)} 
\bigg(\fint_{Q'} |f(x')-f_{Q'}|^2 dx' \bigg)^{\frac12} < \eta. 
\end{align}
If $\ell(Q) \ge r_0$, then $\ell(2^k Q) \ge r_0$ and by \eqref{eq:gaaf-2} 
\begin{align*}
a_k(Q) \le \sup_{Q' \subset \R^{n-1}:\ell(Q') \geq r_0} \bigg(\fint_{Q'} |f(x')-f_{Q'}|^2 dx' \bigg)^{\frac12}  < \eta. 
\end{align*}
If $\ell(Q)<r_0$, we recall that $Q \subset \R^{n-1} \setminus Q(0, r)$ and pick $r \ge 2^{k+1} r_0$ sufficiently large so that $2^k Q \subset \R^{n-1} \setminus Q(0, r/2) \subset \R^{n-1} \setminus Q(0, r_0)$. Thus, this and \eqref{eq:gaaf-3} imply  
\begin{align*}
a_k(Q) \le \sup_{Q' \subset \R^{n-1} \setminus Q(0, r_0)} \bigg(\fint_{Q'} |f(x')-f_{Q'}|^2 dx' \bigg)^{\frac12} < \eta. 
\end{align*}
Collecting these estimates, we have shown that given $\eta>0$, there exists $r_0=r_0(\eta)>0$ such that for any $r \ge 2^{k+1}r_0$, 
\begin{align*}
 \sup_{Q \subset \R^{n-1} \setminus Q(0,r)} a_k(Q) < \eta. 
 \end{align*}
 That is, for any fixed $k \in \N$, \eqref{eq:rr-23} holds. This completes the proof. 
\end{proof}

\section{The uniqueness of the solution}\label{sec:unique} 
In this section, we will prove the uniqueness of the solution, which will follow from Fatou-type theorems \ref{thm:Fatou-CMO} and \ref{thm:Fatou-XMO}. 

\begin{lemma}\label{lem:ff-Car}
Let $L$ be an $N \times N$ elliptic system with constant complex coefficients as in \eqref{eq:Lu}--\eqref{eq:elliptic} and let $P^L$ be the Poisson kernel for $L$ in $\R^n_+$ from Theorem {\rm \ref{thm:Poisson}}, together with $K^L$ as in \eqref{def:KL}. Set 
\[
\Phi(x') := \partial_n K^L(x', 1), \quad x' \in \R^{n-1}. 
\] 
Then for any $f \in L^1(\R^{n-1}, dx'/(1+|x'|^n))^N$ and for any cube $Q \subset \R^{n-1}$, one has 
\begin{align}\label{eq:ff-Car}
\bigg(\fint_Q |f - f_Q|^2 \, dx' \bigg)^{\frac12}
\lesssim \sum_{k=1}^{\infty} 2^{-k} 
\bigg(\frac{1}{|2^k Q|} \iint_{T_{2^k Q}} |\Phi_t*f(x')|^2 \, \frac{dx' dt}{t} \bigg)^{\frac12}, 
\end{align}
where $\Phi_t(x')=t^{1-n}\Phi(x'/t)$ for each $(x', t) \in \R^n_+$. 
\end{lemma}

\begin{proof}
Let $f \in L^1(\R^{n-1}, dx'/(1+|x'|^n))^N$ and $Q \subset \R^{n-1}$ be a cube. By duality, we have 
\begin{align}\label{eq:duality}
\bigg(\fint_Q |f - f_Q|^2 \, dx' \bigg)^{\frac12}
&=\sup_{\|g\|_{L^2(Q)}=1} |Q|^{-\frac12} \bigg|\int_Q  (f-f_Q) g \, dx'\bigg|
\\ \nonumber
&=\sup_{\|g\|_{L^2(Q)}=1} |Q|^{-\frac12} \bigg|\int_Q  f \, (g-g_Q) \, dx'\bigg|. 
\end{align}
In what follows, we fix $g \in L^2(Q)$ with $\|g\|_{L^2(Q)}=1$, and write $h:=(g-g_Q) \mathbf{1}_Q$. Then, 
\begin{align}\label{eq:gh}
\supp h \subset Q, \quad \|h\|_{L^2(Q)} \le 2 \|g\|_{L^2(Q)}, 
\quad\text{and}\quad\int_{\R^{n-1}} h(x') \, dx' = 0. 
\end{align}

To continue, we introduce the matrix-valued functions 
\begin{align*}
\widetilde{\Phi}(x') := \Phi^{\top}(-x'), \quad 
\widetilde{\Psi}_{\varepsilon, \varepsilon^{-1}}(x') := \Psi^{\top}_{\varepsilon, \varepsilon^{-1}}(-x'), \quad
\widetilde{P}^L(x') := (P^L)^{\top}(-x'), \quad x' \in \R^{n-1}, 
\end{align*}
where for all $0<a<b<\infty$, 
\[
\Psi_{a, b}(x') := 4 \int_a^b \Phi_t*\Phi_t(x') \, \frac{dt}{t}, \quad x \in \R^{n-1}. 
\]
Here and elsewhere we use the notation
\[
\langle \lambda, \lambda' \rangle := \sum_{\alpha=1}^N \lambda_{\alpha} \lambda'_{\alpha}, 
\quad\text{for all } \, \lambda=(\lambda_{\alpha})_{1 \le \alpha \le N}, \,\,\, 
\lambda'=(\lambda'_{\alpha})_{1 \le \alpha \le N} \in \mathbb{C}^N.  
\]
As shown in \cite[eq. (4-91)]{MMMM19}, there holds 
\begin{align}\label{eq:limit}
\int_{\R^{n-1}} \langle f(x'), h(x') \rangle dx'
=\lim_{\varepsilon \to 0^+} \int_{\R^{n-1}} \langle f(x'), \widetilde{\Psi}_{\varepsilon, \varepsilon^{-1}}*h(x') \rangle dx'. 
\end{align}
By definition, for every $\varepsilon>0$, we may write 
\begin{align}\label{eq:FH}
\int_{\R^{n-1}} \langle f(x'), \widetilde{\Psi}_{\varepsilon, \varepsilon^{-1}}*h(x') \rangle \, dx' 
&= \int_{\R^{n-1}} \langle \Psi_{\varepsilon, \varepsilon^{-1}}*f(x'), h(x') \rangle \, dx' 
\\ \nonumber
&= \int_{\varepsilon}^{\varepsilon^{-1}} \int_{\R^{n-1}} \langle \Phi_t * \Phi_t*f(x'), h(x') \rangle \, dx' \frac{dt}{t}
\\ \nonumber
&= \int_{\varepsilon}^{\varepsilon^{-1}} \int_{\R^{n-1}} \langle \Phi_t *f(x'), \widetilde{\Phi}_t * h(x') \rangle \, dx' \frac{dt}{t} 
\\ \nonumber
&=: \int_{\varepsilon}^{\varepsilon^{-1}} \int_{\R^{n-1}} \langle F(x', t), H(x', t) \rangle \, dx' \frac{dt}{t}, 
\end{align}
where 
\begin{align*}
F(x', t) := \Phi_t*f(x') \quad\text{and}\quad H(x', t) := \widetilde{\Phi}_t*h(x'), \quad (x', t) \in \R^n_+. 
\end{align*}
Considering the last integral above, we control 
\begin{multline}\label{eq:FHJak}
\iint_{\R^n_+} |\langle F(x', t), H(x', t) \rangle| \frac{dx' dt}{t} 
\\
\le \bigg(\iint_{T_{2Q}} + \sum_{k=1}^{\infty} 
\iint_{T_{2^{k+1}Q} \setminus T_{2^k Q}} \bigg) |F(x', t)| |H(x', t)| \frac{dx' dt}{t} 
=: \mathcal{J}_0 + \sum_{k=1}^{\infty} \mathcal{J}_k. 
\end{multline}

In order to analyze $\mathcal{J}_0$, we observe that 
\begin{align}\label{eq:Phi-PL}
\Phi_t(x') = t^{1-n} (\partial_n K^L)(x'/t, 1)
= t \, (\partial_n K^L)(x', t) = t \partial_t K^L(x', t) = t \partial_t P_t^L(x'), 
\end{align}
and by \eqref{eq:grad-PL} with $t=1$, 
\begin{align*}
\int_{\R^{n-1}} \widetilde{\Phi}(x') \, dx'
&= \bigg(\int_{\R^{n-1}} \Phi(-x') \, dx' \bigg)^{\top} 
= \bigg(\int_{\R^{n-1}} \Phi(x') \, dx' \bigg)^{\top} 
\\
&= \bigg(\int_{\R^{n-1}} \partial_n K^L(x', 1) \, dx' \bigg)^{\top} 
=0, 
\end{align*}
Using these and \eqref{eq:Poisson-smooth}, and setting $\theta(x, t; y') := \widetilde{\Phi}_t(x'-y')$ for every $(x', t) \in \R^n_+$, we obtain 
\begin{align}
\label{eq:H-kernel-1} |\theta(x', t; y')| &= |\Phi_t(y'-x')| = t |\partial_t K^L(y'-x', t)| 
\lesssim \frac{t}{|(x'-y', t)|^n}, 
\\
\label{eq:H-kernel-2} |\nabla_{y'}\theta(x', t; y')| &=  t |(\nabla_{y'} \partial_t K^L)(y'-x', t)| 
\lesssim \frac{t}{|(x'-y', t)|^{n+1}}, 
\\
\label{eq:H-kernel-3} \int_{\R^{n-1}} \theta(x', t; y') \, dy'
&=\int_{\R^{n-1}} \widetilde{\Phi}_t(x'-y') \, dy'
=\int_{\R^{n-1}} \widetilde{\Phi}(y') \, dy' =0. 
\end{align}
This verifies the assumptions \eqref{eq:size}, \eqref{eq:cancel}, and \eqref{eq:smooth}. If we define the matrix-valued integral operator $\Theta$ as in \eqref{eq:Theta}, then \eqref{eq:L2-bound} yields that 
\begin{align}\label{eq:H-bound}
H=\Theta : L^2(\R^{n-1}, \mathbb{C}^N) \to L^2(\R^n_+, dx' dt/t)^N \text{ is bounded}. 
\end{align}
Accordingly, the Cauchy-Schwarz inequality, \eqref{eq:H-bound} and \eqref{eq:gh} imply 
\begin{align}\label{eq:Jak-1}
\mathcal{J}_0 
&\le \bigg(\iint_{T_{2Q}} |F(x', t)|^2 \frac{dx' dt}{t} \bigg)^{\frac12} 
\bigg(\iint_{T_{2Q}} |G(x', t)|^2 \frac{dx' dt}{t} \bigg)^{\frac12} 
\nonumber \\ 
&\le 
\bigg(\iint_{T_{2Q}} |F(x', t)|^2 \frac{dx' dt}{t} \bigg)^{\frac12} 
\bigg(\iint_{\R^n_+} |H(x', t)|^2 \frac{dx' dt}{t} \bigg)^{\frac12} 
\nonumber \\ 
&\lesssim \bigg(\iint_{T_{2Q}} |\Phi_t*f(x')|^2 \frac{dx' dt}{t} \bigg)^{\frac12}. 
\end{align}

To estimate $\mathcal{J}_k$ for $k \ge 1$, we claim that 
\begin{align}\label{eq:H-pointwise}
|H(x', t)| \lesssim \frac{t \, \ell(Q)}{(2^k \ell(Q))^{n+1}} \|h\|_{L^1(Q)}, 
\quad \forall (x', t) \in T_{2^{k+1}Q} \setminus T_{2^k Q}. 
\end{align}
Indeed, using the notation above, the last estimate in \eqref{eq:gh}, and \eqref{eq:H-kernel-2}, we arrive that 
\begin{align*}
|H(x', t)| &= \bigg|\int_{Q} (\theta(x', t; y') - \theta(x', t; x'_Q) ) h(y') \, dy'\bigg|
\\ 
&= \sup_{y' \in Q} |\nabla_{y'} \theta(x', t; y')| \, \|h\|_{L^1(Q)} 
\\
&\lesssim \sup_{y' \in Q} \frac{t}{|(x'-y', t)|^{n+1}} \, \|h\|_{L^1(Q)} 
\\
& \lesssim \frac{t \, \ell(Q)}{(2^k \ell(Q))^{n+1}} \|h\|_{L^1(Q)}, 
\end{align*}
whenever $(x', t) \in T_{2^{k+1}Q} \setminus T_{2^k Q}$ and $k \ge 1$.  Hence, the Cauchy-Schwarz inequality and \eqref{eq:H-pointwise} immediately give 
\begin{align}\label{eq:Jak-2}
\mathcal{J}_k 
&\le \bigg(\iint_{T_{2^{k+1}Q} \setminus T_{2^k Q}} |F(x', t)|^2  |H(x', t)|  \frac{dx' dt}{t} \bigg)^{\frac12} 
\bigg(\iint_{T_{2^{k+1}Q} \setminus T_{2^k Q}} |H(x', t)| \frac{dx' dt}{t} \bigg)^{\frac12} 
\nonumber \\ 
&\lesssim 2^{-k} \|h\|_{L^1(Q)} \bigg(\frac{1}{|2^{k+1} Q|}\iint_{T_{2^{k+1}Q}} |\Phi_t*f(x')|^2 \frac{dx' dt}{t} \bigg)^{\frac12} 
\nonumber \\ 
&\lesssim 2^{-k} |Q|^{\frac12} \bigg(\frac{1}{|2^{k+1} Q|}\iint_{T_{2^{k+1}Q}} |\Phi_t*f(x')|^2 \frac{dx' dt}{t} \bigg)^{\frac12}. 
\end{align}
As a consequence, the estimate \eqref{eq:ff-Car} follows at one from \eqref{eq:duality}--\eqref{eq:FHJak}, \eqref{eq:Jak-1}, and \eqref{eq:Jak-2}. 
\end{proof}

Next, we study the boundary behavior of the vertical shifts of a smooth null-solution of $L$ which satisfies a Carleson measure condition in the upper half-space. Moreover, we prove that each vertical shift of a such null-solution, denoted by $u_{\varepsilon}$, has a Poisson integral representation formula, but also the boundary trace of $u_{\varepsilon}$ belongs to $\CMO$ or $\XMO$ uniformly, provided $u \in \V_2(\R^n_+)$ or $u \in \V_3(\R^n_+)$, respectively. 

\begin{lemma}\label{lem:u-ue}
Let $L$ be an $N \times N$ elliptic system with constant complex coefficients as in \eqref{eq:Lu}--\eqref{eq:elliptic} and let $P^L$ be the Poisson kernel for $L$ in $\R^n_+$ from Theorem {\rm \ref{thm:Poisson}}. Suppose that $u \in \C^{\infty}(\R^n_+, \mathbb{C}^N)$ satisfies $Lu=0$ in $\R^n_+$ and $\|u\| < \infty$. For each $\varepsilon>0$, define 
\[
u_{\varepsilon}(x', t) := u(x', t+\varepsilon) \quad\text{and}\quad 
f_{\varepsilon}(x') := u(x', \varepsilon), \quad x' \in \R^n, \, t>0. 
\] 
Then there exists a constant $C \in (0, \infty)$ such that for every $\varepsilon>0$ the following hold: 
\begin{list}{\textup{(\theenumi)}}{\usecounter{enumi}\leftmargin=1cm \labelwidth=1cm \itemsep=0.2cm 
		\topsep=.2cm \renewcommand{\theenumi}{\alph{enumi}}}

\item\label{list:u-ue-1} One has $u_{\varepsilon} \in \C^{\infty}(\overline{\R^n_+}, \mathbb{C}^N)$, $Lu_{\varepsilon}=0$ in $\R^n_+$, and $\|u_{\varepsilon}\|_{\mathcal{C}(\R^n_+)} \le C \|u\|_{\mathcal{C}(\R^n_+)}$. 

\item\label{list:u-ue-2} The Poisson integral representation holds: 
\[
u_{\varepsilon}(x', t) = P_t^L*f_{\varepsilon}(x')
=\big(P_t^L*(\restr{u_{\varepsilon}}{\partial \R^n_+}) \big)(x'), \quad\forall (x', t) \in \R^n_+.
\]

\item\label{list:u-ue-3} The function $f_{\varepsilon}$ belongs to $\dot{\C}^{\eta}(\R^{n-1}, \mathbb{C}^N) \cap \C^{\infty}(\R^{n-1}, \mathbb{C}^N) \cap \BMO(\R^{n-1}, \mathbb{C}^N)$ for each $\eta \in (0, 1)$, and 
\[
\|f_{\varepsilon}\|_{\BMO(\R^{n-1}, \mathbb{C}^N)} \le C \|u\|_{\mathcal{C}(\R^n_+)}.
\] 

\item\label{list:u-ue-4} If $u \in \V_1(\R^n_+)$, then 
\[
\lim_{\varepsilon \to 0^+} \|f - f_{\varepsilon}\|_{\BMO(\R^{n-1}, \mathbb{C}^N)} =0. 
\]

\item\label{list:u-ue-5} For each $j=1,2,3$, $u \in \V_j(\R^n_+)$ implies $u_{\varepsilon} \in \V_j(\R^n_+)$. 

\item\label{list:u-ue-6} If $u \in \V_2(\R^n_+)$, then $f_{\varepsilon} \in \XMO(\R^{n-1}, \mathbb{C}^N)$. 

\item\label{list:u-ue-7} If $u \in \V_3(\R^n_+)$, then $f_{\varepsilon} \in \CMO(\R^{n-1}, \mathbb{C}^N)$. 
\end{list}
\end{lemma}

\begin{proof}
Parts \eqref{list:u-ue-1}--\eqref{list:u-ue-4} are contained in \cite{MMMM19}. To prove part \eqref{list:u-ue-5}, we assume that $|\nabla u(x)|^2 \, dx' dt$ is a Carleson measure in $\R^n_+$, that is, $\|u\|_{\mathcal{C}(\R^n_+)} < \infty$. Fix a cube $Q \subset \R^{n-1}$ and $\varepsilon>0$. Let $r>0$ be an arbitrary number. If $\ell(Q) \ge \varepsilon$, a change of variables yields 
\begin{align}\label{eq:great}
\frac{1}{|Q|} \iint_{T_Q} |\nabla u_{\varepsilon}(x', t)|^2 \, t\, dx' dt
&\le \frac{1}{|Q|} \int_{\varepsilon}^{\ell(Q)+\varepsilon} \int_{Q} |\nabla u(x', t)|^2 \, t\, dx' dt
\nonumber \\ 
&\lesssim \frac{1}{|2Q|} \iint_{T_{2Q}} |\nabla u(x', t)|^2 \, t\, dx' dt, 
\end{align}
which implies that 
\begin{align}\label{eq:bue-2}
\beta_2(u) = 0 \quad \Longrightarrow \quad \beta_2(u_{\varepsilon}) = 0. 
\end{align}
On the other hand, checking the proof of \cite[Lemma 4.3]{MMMM19}, one gets 
\begin{align*}
t |\nabla u(x', t)| \lesssim \bigg(\frac{1}{|Q(x', 2t)|} \iint_{T_{Q(x', 2t)}} |\nabla u(y', s)|^2 \, s\, dy' ds \bigg)^{\frac12}, 
\end{align*}
for any $(x', t) \in \R^n_+$, where $Q(x', 2t)$ denotes the cube in $\R^{n-1}$ centered at $x'$ with side-length $2t$. This gives 
\begin{multline*}
\frac{1}{|Q|} \iint_{T_Q} |\nabla u_{\varepsilon}(x', t)|^2 \, t\, dx' dt
=\frac{1}{|Q|} \iint_{T_Q} |\nabla u(x', t+\varepsilon)|^2 \, t\, dx' dt
\\
\lesssim \frac{1}{|Q|} \iint_{T_Q} \frac{t}{(t+\varepsilon)^2} 
\bigg(\frac{1}{|Q(x', 2t)|}  \iint_{T_{Q(x', 2t)}} |\nabla u(y', s)|^2 \, s\, dy' ds \bigg)^{\frac12} dx' dt, 
\end{multline*}
and furthermore, whenever $\ell(Q) \le \min\{r, \varepsilon\}$,
\begin{multline}\label{eq:bue-11}
\frac{1}{|Q|} \iint_{T_Q} |\nabla u_{\varepsilon}(x', t)|^2 \, t\, dx' dt
\lesssim \int_0^{\ell(Q)}\frac{t}{(t+\varepsilon)^2} dt \sup_{\substack{Q' \subset \R^{n-1} \\ \ell(Q') \le 2r}}
\bigg(\frac{1}{|Q'|}  \iint_{T_{Q'}} |\nabla u(y', s)|^2 \, s\, dy' ds \bigg)^{\frac12} 
\\
\lesssim \sup_{Q' \subset \R^{n-1}: \ell(Q') \le 2r} 
\bigg(\frac{1}{|Q'|}  \iint_{T_{Q'}} |\nabla u(y', s)|^2 \, s\, dy' ds \bigg)^{\frac12},  
\end{multline}
which immediately implies 
\begin{align}
\label{eq:bue-11} \beta_1(u) = 0 & \quad \Longrightarrow \quad \beta_1(u_{\varepsilon}) = 0, 
\\ 
\label{eq:bue-12} \beta_1(u) = \beta'_3(u; Q) = 0 & \quad \Longrightarrow \quad \beta'_3(u_{\varepsilon}; Q) = 0. 
\end{align}

To proceed, let $Q \subset \R^{n-1} \setminus Q(0, r)$ and $r> \sqrt{n} \, \varepsilon$. If $\ell(Q) \le \varepsilon$, then trivially $\ell(Q) \le r$ and by \eqref{eq:bue-11},  
\begin{align}\label{eq:bue-31}
\frac{1}{|Q|} \iint_{T_Q} |\nabla u_{\varepsilon}(x', t)|^2 \, t\, dx' dt
&\lesssim \sup_{Q' \subset \R^{n-1}: \ell(Q') \le 2r} 
\bigg(\frac{1}{|Q'|}  \iint_{T_{Q'}} |\nabla u(y', s)|^2 \, s\, dy' ds \bigg)^{\frac12},  
\end{align}
If $\ell(Q) \ge r/\sqrt{n}$, then necessarily $\ell(Q) \ge \varepsilon$, and by \eqref{eq:great}, 
\begin{align}\label{eq:bue-32}
\frac{1}{|Q|} \iint_{T_Q} |\nabla u_{\varepsilon}(x', t)|^2 \, t\, dx' dt
&\lesssim \sup_{Q' \subset \R^{n-1}: \ell(Q') \ge 2r/\sqrt{n}} 
\bigg(\frac{1}{|Q'|} \iint_{T_{Q'}} |\nabla u(x', t)|^2 \, t \, dx' dt \bigg)^{\frac12}, 
\end{align}
If $\varepsilon < \ell(Q) < r/\sqrt{n}$, then $2Q \subset \R^{n-1} \setminus Q(0, r/2)$ and by \eqref{eq:great} again, 
\begin{align}\label{eq:bue-33}
\frac{1}{|Q|} \iint_{T_Q} |\nabla u_{\varepsilon}(x', t)|^2 \, t\, dx' dt
&\lesssim \sup_{Q' \subset \R^{n-1} \setminus Q(0, r/2)} 
\bigg(\frac{1}{|Q'|} \iint_{T_{Q'}} |\nabla u(x', t)|^2 \, t \, dx' dt \bigg)^{\frac12}. 
\end{align}
With \eqref{eq:bue-31}--\eqref{eq:bue-33} in hand, we obtain 
\begin{align}\label{eq:bue-3}
\beta_1(u) = \beta_2(u) = \beta_3(u) = 0 \quad \Longrightarrow \quad \beta_3(u_{\varepsilon}) = 0. 
\end{align}
Therefore, having established \eqref{eq:bue-2}, \eqref{eq:bue-11}, \eqref{eq:bue-12}, and \eqref{eq:bue-3}, we conclude that $u \in \V_j(\R^n_+)$ implies $u_{\varepsilon} \in \V_j(\R^n_+)$ for each $j=1,2,3$. 

Finally, let us turn to showing parts \eqref{list:u-ue-6} and \eqref{list:u-ue-7}. Suppose that $u \in \V_j(\R^n_+)$, $j=2,3$. By parts \eqref{list:u-ue-3} and \eqref{list:u-ue-5}, there holds 
\begin{align}\label{eq:ueVj}
f_{\varepsilon} \in L^1(\R^{n-1}, dx'/(1+|x'|^n))^N \quad\text{and}\quad u_{\varepsilon} \in \V_j(\R^n_+). 
\end{align}
In light of \eqref{eq:ff-Car}, \eqref{eq:Phi-PL}, and part \eqref{list:u-ue-2}, this in turn implies  
\begin{align}\label{eq:fe-bk}
\bigg(\fint_Q |f_{\varepsilon} - (f_{\varepsilon})_Q|^2 \, dx' \bigg)^{\frac12}
&\lesssim \sum_{k=1}^{\infty} 2^{-k} 
\bigg(\frac{1}{|2^k Q|} \iint_{T_{2^k Q}} |\Phi_t*f_{\varepsilon}(x')|^2 \, \frac{dx' dt}{t} \bigg)^{\frac12} 
\\ \nonumber
&= \sum_{k=1}^{\infty} 2^{-k} 
\bigg(\frac{1}{|2^k Q|} \iint_{T_{2^k Q}} |\partial_t (P_t^L*f_{\varepsilon})(x')|^2 \, t\, dx' dt \bigg)^{\frac12}
\\ \nonumber
&\le \sum_{k=1}^{\infty} 2^{-k} 
\bigg(\frac{1}{|2^k Q|} \iint_{T_{2^k Q}} |\nabla u_{\varepsilon}(x', t)|^2 \, t\, dx' dt \bigg)^{\frac12}
\\ \nonumber
&=: \sum_{k=1}^{\infty} 2^{-k} b_k(Q). 
\end{align}
Observe that 
\begin{align}\label{eq:supTQ}
\sup_{k \in \N} \sup_{Q \subset \R^{n-1}} b_k(Q) 
\lesssim \|u_{\varepsilon}\|_{\mathcal{C}(\R^n_+)}. 
\end{align}
Recall Propositions \ref{pro:XMO} and \ref{pro:CMO}, and the definition of $\V_2(\R^n_+)$ and $\V_3(\R^n_+)$. In order to obtain \eqref{list:u-ue-6} and \eqref{list:u-ue-7}, by \eqref{eq:ueVj}--\eqref{eq:supTQ},  it suffices to show that for any fix $k \in \N_0$,  
\begin{align}
\label{eq:bb-1} \beta_1(u_{\varepsilon}) =0 
&\quad\Longrightarrow\quad \lim_{r \to 0} \sup_{Q \subset \R^{n-1}:\ell(Q) \leq r} b_k(Q) = 0, 
\\
\label{eq:bb-2} \beta_2(u_{\varepsilon}) =0 
&\quad\Longrightarrow\quad \lim_{r \to \infty} \sup_{Q \subset \R^{n-1}:\ell(Q) \ge r} b_k(Q) =0, 
\\
\label{eq:bb-3} \beta'_3(u_{\varepsilon}; Q) =0 
&\quad\Longrightarrow\quad \lim_{|x_Q| \to \infty} b_k(Q) =0, \quad\text{for each cube } Q \subset \R^{n-1},  
\\
\label{eq:bb-23} \beta_2(u_{\varepsilon}) = \beta_3(u_{\varepsilon}) =0 &\quad\Longrightarrow\quad \lim_{r \to \infty} \sup_{Q \subset \R^{n-1} \setminus Q(0,r)} b_k(Q) =0. 
\end{align}
Indeed, \eqref{eq:bb-1}--\eqref{eq:bb-3} hold trivially. To show \eqref{eq:bb-23}, we fix $r>0$ and $Q \subset \R^{n-1} \setminus Q(0, r)$. The fact $\beta_2(u_{\varepsilon})=\beta_3(u_{\varepsilon})=0$ implies that for any given $\eta>0$, there exists $r_0=r_0(\eta)>0$ such that  
\begin{align}
\label{eq:uubb-1} & \sup_{Q' \subset \R^{n-1}:\ell(Q') \geq r_0} 
\bigg(\frac{1}{|Q'|} \iint_{T_{Q'}} |\nabla u_{\varepsilon}(x', t)|^2 \, t\, dx' dt \bigg)^{\frac12} < \eta,
\\
\label{eq:uubb-2} & \sup_{Q' \subset \R^{n-1} \setminus Q(0, r_0)} 
\bigg(\frac{1}{|Q'|} \iint_{T_{Q'}} |\nabla u_{\varepsilon}(x', t)|^2 \, t\, dx' dt \bigg)^{\frac12} < \eta. 
\end{align}
If $\ell(Q) \ge r_0$, then $\ell(2^k Q) \ge r_0$ and by \eqref{eq:uubb-1} 
\begin{align}\label{eq:bQTQ-1}
b_k(Q) \le \sup_{Q' \subset \R^{n-1}:\ell(Q') \geq r_0} 
\bigg(\frac{1}{|Q'|} \iint_{T_{Q'}} |\nabla u_{\varepsilon}(x', t)|^2 \, t\, dx' dt \bigg)^{\frac12}  < \eta. 
\end{align}
If $\ell(Q)<r_0$, we recall that $Q \subset \R^{n-1} \setminus Q(0, r)$ and pick $r \ge 2^{k+1} r_0$ sufficiently large so that $2^k Q \subset \R^{n-1} \setminus Q(0, r/2) \subset \R^{n-1} \setminus Q(0, r_0)$. Thus, this and \eqref{eq:gaaf-3} imply  
\begin{align}\label{eq:bQTQ-2}
b_k(Q) \le \sup_{Q' \subset \R^{n-1} \setminus Q(0, r_0)} 
\bigg(\frac{1}{|Q'|} \iint_{T_{Q'}} |\nabla u_{\varepsilon}(x', t)|^2 \, t\, dx' dt \bigg)^{\frac12} < \eta. 
\end{align}
Gathering \eqref{eq:bQTQ-1} and \eqref{eq:bQTQ-2}, we conclude that given $\eta>0$, there exists $r_0=r_0(\eta)>0$ such that for any $r \ge 2^{k+1}r_0$, 
\begin{align*}
 \sup_{Q \subset \R^{n-1} \setminus Q(0,r)} b_k(Q) < \eta. 
 \end{align*}
 Therefore, for any fixed $k \in \N$, \eqref{eq:bb-23} holds. This completes the proof. 
\end{proof}

\begin{proof}[\bf Proof of Theorems \ref{thm:Fatou-CMO} and \ref{thm:Fatou-XMO}] 
Let us first show Theorem \ref{thm:Fatou-CMO}. Assume that $u \in {\rm LMO}(\R^n_+) \cap \V_3(\R^n_+)$. By definition, it immediately implies that 
\begin{align}\label{eq:van-Car}
\text{$|\nabla u(x', t)|^2 t \, dx' dt$ is a Carleson measure in $\R^n_+$},
\end{align} 
that is, $\|u\|_{\mathcal{C}(\R^n_+)}<\infty$, which together with \cite[Theorem 1.2]{MMMM19} yields that  
\begin{equation*}
f:=\restr{u}{\partial \R^n_+}^{\rm n.t.} \text{ exists a.e. in } \R^{n-1}, \text{ and lies in } \BMO(\R^{n-1}, \mathbb{C}^N). 
\end{equation*}
Thus, to get \eqref{eq:u-CMO}, by definition and Proposition \ref{pro:CMO}, it suffices to show 
\begin{align}
\label{eq:gaf-1} \gamma_1(f) &:= \lim_{r \to 0^+} \sup_{Q \subset \R^{n-1}:\ell(Q) \le r} \fint_Q |f(x')-f_Q| dx' =0,
\\ 
\label{eq:gaf-2} \gamma_2(f) &:= \lim_{r \to \infty} \sup_{Q \subset \R^{n-1}:\ell(Q) \ge r} \fint_Q |f(x')-f_Q| dx' =0,
\\
\label{eq:gaf-3} \gamma_3(f) &:= \lim_{r \to \infty} \sup_{Q \subset \R^{n-1} \setminus Q(0,r)} \fint_Q |f(x')-f_Q| dx' =0. 
\end{align}

To proceed, for each $\varepsilon>0$, define 
\[
u_{\varepsilon}(x', t) := u(x', t+\varepsilon) \quad\text{and}\quad 
f_{\varepsilon}(x') := u(x', \varepsilon), \quad x' \in \R^n, \, t>0. 
\] 
Invoking Lemma \ref{lem:u-ue} and \eqref{eq:van-Car}, we get 
\begin{align}\label{eq:fe-CMO}
\lim_{\varepsilon \to 0^+} \|f_{\varepsilon} - f\|_{\BMO(\R^{n-1}, \mathbb{C}^N)} =0 \quad\text{and}\quad 
f_{\varepsilon} \in \CMO(\R^{n-1}, \mathbb{C}^N), \quad\forall \varepsilon>0.  
\end{align}
Note that for every cube $Q \in \R^{n-1}$, 
\begin{align*}
\fint_Q |f - f_Q| \, dx'
&\le \fint_Q |(f-f_{\varepsilon}) - (f-f_{\varepsilon})_Q| \, dx'
+ \fint_Q |f_{\varepsilon} - (f_{\varepsilon})_Q| \, dx'
\\
&\le \|f_{\varepsilon} - f\|_{\BMO(\R^{n-1}, \mathbb{C}^N)} 
+ \fint_Q |f_{\varepsilon} - (f_{\varepsilon})_Q| \, dx', 
\end{align*}
which combining with the second estimate in \eqref{eq:fe-CMO} gives 
\begin{align*}
\gamma_j(f) 
&\le \|f_{\varepsilon} - f\|_{\BMO(\R^{n-1}, \mathbb{C}^N)}, \quad \forall \varepsilon>0, \, \, j=1,2,3. 
\end{align*}
Letting $\varepsilon \to 0^+$ above and using the first estimate in \eqref{eq:fe-CMO}, we conclude \eqref{eq:gaf-1}--\eqref{eq:gaf-3}. 

Let us turn to the proof of \eqref{eq:CMO-V3}. Indeed, the left-to-right inclusion is contained in Proposition \ref{pro:existence}, while the right-to-left inclusion is a consequence of \eqref{eq:u-CMO}. 

Finally, the proof of Theorems \ref{thm:Fatou-XMO} is the same as above, since the corresponding estimates have been established in Proposition \ref{pro:existence} and Lemma \ref{lem:u-ue}. 
\end{proof}

Finally, let us see how to conclude our main theorems \ref{thm:CMO} and \ref{thm:XMO} from Proposition \ref{pro:existence}, Theorem \ref{thm:Fatou-CMO}, and Theorem \ref{thm:Fatou-XMO}. 

\medskip 
\begin{proof}[\bf Proof of Theorems \ref{thm:CMO} and \ref{thm:XMO}]
From Proposition \ref{pro:existence}, we see that the function $u$ defined in \eqref{eq:CMO-solution} solves the $\CMO$-Dirichlet boundary value problem \eqref{eq:CMO}. The uniqueness is a consequence of Theorem \ref{thm:Fatou-CMO}. The proof of Theorem \ref{thm:XMO} follows the same scheme. 
\end{proof}

\end{document}